\documentclass[12pt]{amsart}
\usepackage{epsfig,comment}
\usepackage{mathtools}
\usepackage[subnum]{cases}
\usepackage{enumerate}
\usepackage{bbm,bm}
\usepackage{amsmath,amssymb,mathrsfs}
\usepackage{color}
\newtheorem{theorem}{Theorem}[section]

\newtheorem{lemma}[theorem]{Lemma}
\newtheorem{corollary}[theorem]{Corollary}
\newtheorem{remark}[theorem]{Remark}

\newcommand{\R}{\mathbb{R}} 
\newcommand{\E}{\mathbb{E}} 

\newcommand{\NN}{N^{-(b-1+\frac{2}{n-1})}}
\newcommand{\prob}{\mathbb{P}}

\DeclareMathOperator{\dist}{dist}

\DeclareMathOperator{\vol}{vol}
\DeclareMathOperator{\interior}{int}
\DeclareMathOperator{\aff}{aff}

\begin{document}

\title[Expected $T$-functional of random polytopes]{Asymptotic expected $T$-functionals of random polytopes\\ with applications to $L_p$ surface areas}

\author{Steven Hoehner}
\address{Department of Mathematics \& Computer Science, Longwood University, U.S.A.}
\email{hoehnersd@longwood.edu}

\author{Ben Li}
\address{Department of Mathematical Methods in Physics, Warsaw University, Poland.}
\email{liben@fuw.edu.pl}

\author{Michael Roysdon}
\address{Institute for Computational and Experimental Research in Mathematics, U.S.A. and Department of Pure Mathematics, Tel Aviv University, Israel.}
\email{michael\_roysdon@brown.edu}

\author{Christoph Th\"ale}
\address{Faculty of Mathematics, Ruhr University Bochum, Germany.}
\email{christoph.thaele@rub.de}

    \date{\today}

	
\maketitle

\begin{abstract}
An asymptotic formula is proved for the expected $T$-functional of the convex hull of independent and identically distributed random  points sampled from the Euclidean unit sphere in $\R^n$ according to an arbitrary positive continuous density. 
As an application, the approximation of the sphere by random polytopes in terms of $L_p$ surface area differences is discussed. 
\end{abstract}



\section{Introduction and main results}

\subsection{Introduction}

Random polytopes are a cornerstone of stochastic and integral geometry. They connect convex geometry and probability, and have a number of applications in computer science, statistics and machine learning theory. There is a vast literature on the subject of random polytopes, and we refer the reader to the survey articles \cite{BaranySurvey,HugSurvey,ReitznerSurvey} and the references therein.

Random polytopes generated as the convex hull of independent and identically distributed (i.i.d.) points sampled from the Euclidean unit ball or the Euclidean unit sphere in $\R^n$ according to the respective uniform distribution are important models that have been studied extensively in stochastic geometry. For the Euclidean ball, Wieacker \cite{WieackerThesis} extended the results of R\'enyi and Sulanke \cite{RenyiSulanke} to $\R^n$. He obtained asymptotic formulas for the expected volume and expected surface area of a random polytope inscribed in a ball, and introduced what is known today as the \emph{$T$-functional} $T_{a,b}^{n,k}(Q)$ of a polytope $Q$ in $\R^n$. It is defined as
\[
T_{a,b}^{n,k}(Q) := \sum_{F\in\mathcal{F}_k(Q)}\dist(o,\aff(F))^a \vol_k(F)^b
\]
where $a,b\geq 0$ are parameters, $\mathcal{F}_k(Q)$ stands for the set of $k$-faces of $Q$ for $k\in\{0,1,\ldots,n\}$, $o$ stands for the origin of $\R^n$ and $\dist(o,\aff(F))=\min\{\|x\|:x\in \aff(F)\}$ is the Euclidean distance from $o$ to the affine hull $\aff(F)$ of $F$. In particular, we notice that $T_{0,0}^{n,k}(Q)$ is the number of $k$-dimensional faces of $Q$ and $T_{0,1}^{n,k}(Q)$ is the $k$-content of the union of all $k$-faces of $Q$, whereas ${1\over n}T_{1,1}^{n,n-1}(Q)$ coincides with the volume of $Q$ as long as $Q$ contains the origin in its interior. Later on, Affentranger \cite{Affentranger1991}  described for $k=n-1$ and general parameters $a,b\geq 0$, the asymptotic behavior, as $N\to\infty$, of the expected $T$-functional of the convex hull of $N$ i.i.d.\ random points distributed according to a so-called beta distribution in the $n$-dimensional unit ball (so-called beta polytopes). 
We recall that the beta distribution with parameter $\beta>-1$ in the $n$-dimensional Euclidean unit ball $B_n$ has Lebesgue density
$$
f_{n,\beta}(x):=c_{n,\beta}(1-\|x\|^2)^\beta\mathbbm{1}_{\{x:\|x\|<1\}},\qquad c_{n,\beta}:={\Gamma({n\over 2}+\beta+1)\over\pi^{\frac{n}{2}}\Gamma(\beta+1)}.
$$
For example, choosing $\beta=0$ we obtain the uniform distribution on $B_n$, while the weak limit as $\beta\to -1^+$ corresponds to the uniform distribution on $\partial B_n$, the $(n-1)$-dimensional unit sphere. An exact formula for the expected $T$-functional $\mathbb{E}[T_{a,b}^{n,n-1}(P_{n,N}^\beta)]$ of a beta polytope $P_{n,N}^\beta$, defined as the convex hull of $N\geq n+1$ i.i.d.\ random points distributed with respect to the density $f_{n,\beta}$, was provided by Kabluchko, Temesvari and Th\"ale \cite{KabluchkoEtAl2019}:
\begin{align*}
    \mathbb{E}[T_{a,b}^{n,n-1}(P_{n,N}^\beta)] = C_{n,N}^{\beta,b}\int_{-1}^1|h|^a(1-h^2)^{n\beta-{n-1\over 2}(n+b+1)}F_{1,\beta}(h)^{N-n}\,dh,
\end{align*}
where
\begin{align*}
F_{1,\beta}(h) &:=\int_{-1}^h f_{1,\beta}(x)\,dx,\qquad
C_{n,N}^{\beta,b} := {N\choose n}n!\vol_n(B_n)\mathbb{E}_{\beta}[\mathcal{V}_{n,n-1}^{b+1}]\Big({c_{n,\beta}\over c_{n-1,\beta}}\Big)^n
\end{align*}
and $\mathbb{E}_{\beta}[\mathcal{V}_{n,n-1}^{b+1}]$ is the moment of order $b+1$ of the volume of the $(n-1)$-dimensional random beta simplex $P_{n-1,n}^\beta$, whose value can be expressed in terms of gamma functions (see \cite[Proposition 2.8]{KabluchkoEtAl2019}). Formally putting $\beta=-1$, this formula also covers the case of the convex hull of $N\geq n+1$ i.i.d.\ uniform random points on the boundary $\partial B_n$ of $B_n$. The expected $T$-functional has further been determined explicitly for beta' polytopes \cite{KabluchkoEtAl2019} and Gaussian polytopes \cite{HugMunsoniusReitzner}, as well as for beta-star polytopes \cite{GKT-2021} and convex hulls of Poisson point processes \cite{KMTT-2019}. The corresponding results have also found applications to stochastic geometry models in spherical and hyperbolic spaces.

In this article, we determine  precise asymptotic formulas for the expected $T$-functional of an inscribed random polytope generated  by i.i.d.\ points selected according to a general continuous positive density function on the Euclidean unit sphere or unit ball in $\mathbb{R}^n$. We focus on the case $k=n-1$ of the $T$-functional, in which case the summation in the definition ranges over the set of facets $\mathcal{F}_{n-1}(Q)$ of $Q$. For a polytope $Q$ in $\R^n$, we thus denote $T_{a,b}(Q):=T_{a,b}^{n,n-1}(Q)$. As already pointed out above, the case $k=n-1$ gives rise to a number of important functionals of polytopes. For example:
\begin{itemize}
    \item $T_{0,0}(Q)=|\mathcal{F}_{n-1}(Q)|$ is the number of facets of $Q$;

    \item $T_{0,1}(Q)=\mu_{\partial Q}(\partial Q)$ is the surface area (i.e.,  $(n-1)$-dimensional Hausdorff measure) of $Q$;

    \item $\tfrac{1}{n}T_{1,1}(Q)=\vol_n(Q)$ is the volume of $Q$, if $Q$ contains the origin in its interior.
\end{itemize}  
To this list we can also add what is known as the \textit{$L_p$ surface area} of a polytope $Q$ containing the origin in its interior, which for  $p\in\R$ is defined as
\begin{equation}\label{eq:SpQ}
    S_p(Q):=T_{1-p,1}(Q)=\sum_{F\in\mathcal{F}_{n-1}(Q)}\dist(o,\aff(F))^{1-p}\vol_{n-1}(F).
\end{equation}
For $p>1$, Lutwak \cite{Lutwak93} defined the $L_p$ surface area measure of a general convex body  in $\R^n$ which contains the origin in its interior, and formulated and studied the famous  {$L_p$ Minkowski problem}. This problem asks for necessary and sufficient conditions on a finite Borel measure $\mu$ on the sphere $\partial B_n$ so that $\mu$ is the $L_p$ surface area of some convex body $K$ in $\R^n$. The important special case when $K$ is a polytope is called the discrete $L_p$ Minkowski problem, and for more background we refer the reader to \cite{HLYZ2005, Stancu2002, Stancu2003, ZhuIUMJ} and the references therein. 

\medspace

The remaining parts of this paper are structured as follows. In Section \ref{sec:Tfunctional}, we present our general result on the expected $T$-functional. In Section \ref{sec:lpsurfacearea}, we discuss the special case of the $L_p$ surface area, followed by an application to best approximation {in Section \ref{sec:BestApprox}}. Section \ref{sec:lemmas} collects some preliminary geometric lemmas, and in Section \ref{sec:ProofTfunctional} we present the proof of our main result.

\subsection{Main results for the expected $T$-functional}\label{sec:Tfunctional}

Our main result is an asymptotic description of the expected $T$-functional of the convex hull of $N$ i.i.d.\ random points distributed according to an arbitrary positive and continuous density function on the sphere, as $N\to\infty$, for arbitrary parameter values $a,b\geq 0$. We would like to highlight that such a result is distinguished from those in \cite{Affentranger1991,KabluchkoEtAl2019,Muller1990} as we consider densities which are not necessarily rotationally invariant. In fact, {choosing the uniform density}, $k=n-1$ and $a=b=1$, our result reduces to that of \cite[Theorem 2]{Muller1990} or \cite[Theorem 5]{Affentranger1991} (choosing $i=0$ there), while the exact formula for the expected $T$-functional of a random beta polytope with general parameters $a,b\geq 0$ can be found in \cite[Theorem 2.13]{KabluchkoEtAl2019}.

In what follows, for $k\in\mathbb{N}$ we shall use  the notation $\mu_{\partial B_k}$ to denote the $(k-1)$-dimensional spherical Lebesgue measure on the $(k-1)$-dimensional unit sphere $\partial B_k$, and we indicate the $k$-dimensional Lebesgue measure by $\vol_k$. 

\begin{theorem}\label{T-Thm}
Let $Q_{n,N}^f$ denote the convex hull of $N\geq n+1$ random points chosen independently according to a probability distribution, which has a positive continuous density function $f$ with respect to $\mu_{\partial B_n}$ on $\partial B_n$. 
Then for any fixed $a,b\geq 0$,
\begin{align*}
    \E [T_{a,b}(Q_{n,N}^f)] = \frac{(n-1)^{n-1}\E[\mathcal{V}_{n,n-1}^{b+1}]}{n\vol_{n-1}(B_{n-1})^{b-1}}N^{-(b-1)}&\left(c_1(n,b,f)
    -c_2(n,a,b,f)N^{-\frac{2}{n-1}}\right)\\&\times\left(1+O\left(N^{-\frac{2}{n-1}}\right)\right)
\end{align*}
as $N\to\infty$, where
\begin{align*}
    c_1(n,b,f)&:=\Gamma(n+b-1)\int_{\partial B_n}f(x)^{1-b}\, d\mu_{\partial B_n}(x),\\
    c_2(n,a,b,f)&:=\frac{1}{2}\left(a+\frac{(n-1)(n+b-2)}{n+1}\right)\Gamma\left(n+b-1+\tfrac{2}{n-1}\right)\frac{\int_{\partial B_n}f(x)^{1-b-\frac{2}{n-1}}\,d\mu_{\partial B_n}(x)}{\vol_{n-1}(B_{n-1})^{\frac{2}{n-1}}}
\end{align*}
and $\E[\mathcal{V}_{n,n-1}^{b+1}]$ is the $(b+1)$st moment of the $(n-1)$-dimensional volume of the random simplex spanned by $n$ i.i.d.\ uniform random points on the $(n-2)$-dimensional unit sphere $\partial B_{n-1}$.
\end{theorem}

\begin{remark}\rm 
It is interesting to observe that the rate in $N$ only depends on $b$ and does not involve the parameter $a$. The latter only appears in the second-order term.
\end{remark}

\begin{remark}\rm
{One can also derive an asymptotic formula  for the expected $T$-functional of a random polytope generated as the convex hull of $N\geq n+1$ random points chosen independently according to a probability distribution which has a positive continuous density function with respect to the Lebesgue measure on $B_n$. Since the proof is a straightforward adaptation of that of Theorem \ref{T-Thm}, we refrain from presenting the details.}
\end{remark}

\subsection{The $L_p$ surface area difference}\label{sec:lpsurfacearea}

In this section, we specialize Theorem \ref{T-Thm} to discuss the \emph{$L_p$ surface area difference} $\Delta_{S_p}(B_n,Q_{n,N}^f)$ of the ball and a random inscribed polytope $Q_{n,N}^f$ that contains the origin in its interior. It is defined as
$$
\Delta_{S_p}(B_n,Q_{n,N}^f):=\mu_{\partial B_n}(\partial B_n)-S_p(Q_{n,N}^f),
$$
where we recall the definition of the $L_p$ surface area from \eqref{eq:SpQ}. In fact, choosing $a=1-p$ with $p\in(-\infty,1]$ and $b=1$, we arrive at the following result. 

\begin{corollary}\label{cor:Lpsurface}
Let $Q_{n,N}^f$ denote the convex hull of $N\geq n+1$ random points chosen independently according to a probability distribution, which has a positive continuous density function $f$ with respect to $\mu_{\partial B_n}$ on $\partial B_n$. Then for every fixed $p\in(-\infty,1]$,
$$
\lim_{N\to\infty}N^{\frac{2}{n-1}} \E[\Delta_{S_p}(B_n,Q_{n,N}^f)] = \frac{1}{2}\left(1-p+\frac{(n-1)^2}{n+1}\right)
    \frac{\int_{\partial B_n}f(x)^{-\frac{2}{n-1}}\,d\mu_{\partial B_n}(x)}{\vol_{n-1}(B_{n-1})^{\frac{2}{n-1}}}\cdot\frac{\Gamma\left(n+\frac{2}{n-1}\right)}{(n-1)!}.
$$
\end{corollary}

\begin{remark}\rm 
When $p\in[0,1]$, Corollary \ref{cor:Lpsurface} may be viewed as an \emph{$L_p$ interpolation} of the classical results of  M\"uller \cite{Muller1990} on the expected volume difference and expected surface area difference (see also \cite{Reitzner2002, SW2003}).
\end{remark}

Let $f_{\rm unif}:=\mu_{\partial B_n}(\partial B_n)^{-1}\mathbbm{1}_{\partial B_n}$ denote the uniform density on $\partial B_n$. Following \cite{SW2003}, we show that the uniform density minimizes the constant in the right-hand side of Corollary \ref{cor:Lpsurface}{, which we denote by $c_{\rm bd}(n,p,f)$.} 
Since $\int_{\partial B_n}f(x)\,d\mu_{\partial B_n}(x)=1$, we can write
\begin{align*}
    \left(\int_{\partial B_n}f(x)^{-\frac{2}{n-1}}\,d\mu_{\partial B_n}(x)\right)^{\frac{n-1}{n+1}}
    &=  \left(\int_{\partial B_n}\left(f(x)^{-\frac{2}{n+1}}\right)^{\frac{n+1}{n-1}}\,d\mu_{\partial B_n}(x)\right)^{\frac{n-1}{n+1}}\\
    &\qquad\qquad\times  \left(\int_{\partial B_n}\left(f(x)^{\frac{2}{n+1}}\right)^{\frac{n+1}{2}}\,d\mu_{\partial B_n}(x)\right)^{\frac{2}{n+1}}.
\end{align*}
Applying  H\"older's inequality to the functions $f_1(x)=f(x)^{-\frac{2}{n+1}}$, $f_2(x)=f(x)^{\frac{2}{n+1}}$, and with the exponents $p'=\frac{n+1}{n-1}$ and $q'=\frac{n+1}{2}$, we derive that
\begin{align*}
   \left(\int_{\partial B_n}f(x)^{-\frac{2}{n-1}}\,d\mu_{\partial B_n}(x)\right)^{\frac{n-1}{n+1}} \geq \mu_{\partial B_n}(\partial B_n).
\end{align*}
Therefore,
\begin{align*}
      \int_{\partial B_n}f(x)^{-\frac{2}{n-1}}\,d\mu_{\partial B_n}(x)&\geq \mu_{\partial B_n}(\partial B_n)^{\frac{n+1}{n-1}}
      =\int_{\partial B_n}f_{\rm unif}(x)^{-\frac{2}{n-1}}\,d\mu_{\partial B_n}(x).
\end{align*}
It follows that for any density $f$ on the unit sphere $\partial B_n$  and any $p\in(-\infty,1]$,
\begin{equation}
    c_{\rm bd}(n,p,f) \geq c_{\rm bd}(n,p,f_{\rm unif}),
\end{equation}
which means that the minimizing density is the uniform one, independently of $p$. Hence for every $p\in(-\infty,1]$ and any positive continuous density  $f$ on  $\partial B_n$, we have  
\begin{equation}
     \E[S_p(Q_{n,N}^f)] \leq \E[S_p(Q_{n,N}^{f_{\rm unif}})]
\end{equation}
as $N\to\infty$. In an asymptotic sense, this is a necessary condition for the random variable $S_p(Q_{n,N}^{f_{\rm unif}})$ to  second-order stochastically dominate   $S_p(Q_{n,N}^f)$.

{\begin{remark}\rm
We can rephrase the last inequality in terms of the expected $T$-functional with $b=1$ and arbitrary parameter $a\geq 0$, 
namely, $\E[T_{a,1}(Q_{n,N}^f)] \leq \E[T_{a,1}(Q_{n,N}^{f_{\rm unif}})]$ as $N\to\infty$.
\end{remark}}

\subsection{Relation to best approximation}\label{sec:BestApprox}

For $N\geq n+1$, let $\mathscr{P}_{n,N}^{\rm in}$ denote the set of all polytopes inscribed in $B_n$ with at most $N$ vertices and which contain the origin in their interiors. It follows from a compactness argument that for any fixed $p\in[0,1]$,  there exists a \emph{best-approximating polytope} which achieves the minimum $L_p$ surface area difference
\[
\min_{Q\in\mathscr{P}_{n,N}^{\rm in}}\Delta_{S_p}(B_n,Q).
\]
Moreover, if $Q\in\mathscr{P}_{n,N}^{\rm in}$ and  $p\in[0,1]$, then by the cone-volume formula we have 
\[
n\vol_n(Q) \leq S_p(Q) \leq \mu_{\partial Q}(\partial Q),
\]
where $\mu_{\partial Q}(\partial Q)$ denotes the $(n-1)$-dimensional Hausdorff measure of $\partial Q$. Hence,
\begin{equation}\label{Sp-ineq-1}
\mu_{\partial B_n}(\partial B_n)-\mu_{\partial Q}(\partial Q) \leq \Delta_{S_p}(B_n,Q) \leq n\vol_n(B_n\triangle Q)
\end{equation}
where $B_n\triangle Q=(B_n\cup Q)\setminus(B_n\cap Q)$ denotes the symmetric difference of $B_n$ and $Q$. 

Interestingly, in high dimensions the random approximation of smooth convex bodies is asymptotically as good as the best approximation as $N\to\infty$, up to absolute constants; 
see \cite{Affentranger1991, BH-2022,BHK, BoroczkyCsikos, BoroczkyReitzner04, HK-DCG, HSW, Kur2017, LSW, Muller1990, Reitzner2002, SW2003}. Choosing the minimizing density $f=f_{\rm unif}$ in Corollary \ref{cor:Lpsurface}, by Stirling's inequality we derive
\begin{align*}
\limsup_{N\to\infty}N^{\frac{2}{n-1}}\min_{Q\in\mathscr{P}_{n,N}^{\rm in}}\Delta_{S_p}(B_n,Q) &\leq \limsup_{N\to\infty}N^{\frac{2}{n-1}}\E[\Delta_{S_p}(B_n,Q_{n,N}^{f_{\rm unif}})]\\
&\leq \frac{1}{2}(n-p)\mu_{\partial B_n}(\partial B_n)\left(1+O\left(\frac{\ln n}{n}\right)\right).
\end{align*}
An interesting open question is to determine the optimal lower constant $\tilde{c}_{\rm bd}(n,p,f)$ (up to some factor $o_n(1)$) in the asymptotic lower bound
\[
\liminf_{N\to\infty}N^{\frac{2}{n-1}}\min_{Q\in\mathscr{P}_{n,N}^{\rm in}}\Delta_{S_p}(B_n,Q) \geq \tilde{c}_{\rm bd}(n,p,f).
\]
It follows from \eqref{Sp-ineq-1} and \cite[Thm. 1 i)]{BHK} that $\tilde{c}_{\rm bd}(n,p,f)\geq \frac{1}{2}(n-1)\mu_{\partial B_n}(\partial B_n)\left(1+O\left(\frac{\ln n}{n}\right)\right)$.

\section{Preliminary geometric lemmas}\label{sec:lemmas}

\subsection{Spherical Blaschke-Petkantschin formula}

In this section, we collect a number of geometric lemmas which will be used in the proof of Theorem \ref{T-Thm} below. The first ingredient we need is the spherical Blaschke-Petkantschin formula from \cite[Theorem 4]{Miles1971}, which has later found a far-reaching extension in \cite{Zahle1990}. 

\begin{lemma}\label{BPext}
For $n\geq 2$, let $g:(\partial B_n)^n\to\mathbb{R}$ be a nonnegative measurable function. Then
\begin{align*}
    &\int_{\partial B_n}\cdots \int_{\partial B_n} g(x_1,\ldots,x_n)\, d\mu_{\partial B_n}(x_1)\ldots d\mu_{\partial B_n}(x_n) \\
    &= (n-1)!\int_{\partial B_n}\int_0^1\int_{\partial B_n\cap H}\cdots\int_{\partial B_n\cap H}g(x_1,\ldots,x_n)\vol_{n-1}([x_1,\ldots,x_n])\\
    &\qquad\qquad\qquad\qquad\qquad\times(1-h^2)^{-{n\over 2}}\,d\mu_{\partial B_n\cap H}(x_1)\ldots d\mu_{\partial B_n\cap H}(x_n)\,dh\,d\mu_{\partial B_n}(u),
\end{align*}
where $H=H(u,h)=\{x\in\mathbb{R}^n:\langle x,u\rangle=h\}$ is the hyperplane orthogonal to $u\in\partial B_n$ at distance $h$ from the origin.
\end{lemma}

\subsection{Geometry of spherical caps}

The next result is a useful estimate of Sch\"utt and Werner from \cite[Lemma 3.12]{SW2003} involving the surface area and radius of a cap of the Euclidean ball.

\begin{lemma}\label{capradiusest}
Let $S$  and $r$ denote the surface area and radius, respectively, of a cap of the Euclidean unit ball in $\R^n$. There exists an absolute constant $C>0$ such that
\begin{align*}
    &\left(\frac{S}{\vol_{n-1}(B_{n-1})}\right)^{\frac{1}{n-1}}-\frac{1}{2(n+1)}\left(\frac{S}{\vol_{n-1}(B_{n-1})}\right)^{\frac{3}{n-1}}-C\left(\frac{S}{\vol_{n-1}(B_{n-1})}\right)^{\frac{5}{n-1}}
    \leq r  \\
    &\leq\left(\frac{S}{\vol_{n-1}(B_{n-1})}\right)^{\frac{1}{n-1}}-\frac{1}{2(n+1)}\left(\frac{S}{\vol_{n-1}(B_{n-1})}\right)^{\frac{3}{n-1}}+C\left(\frac{S}{\vol_{n-1}(B_{n-1})}\right)^{\frac{5}{n-1}}.
\end{align*}
\end{lemma}


Each pair of $u\in\partial B_n$ and $h\in[0,1]$ determines two caps of the sphere, one for each halfspace {determined by the hyperplane $H=u^\perp+hu$}. We select the halfspace $H^-$ that does not contain the polytope and focus on the cap $\partial B_n\cap H^-$. Then we let 
\[
s:=\prob_f(\partial B_n\cap H^-)=\int_{\partial B_n\cap H^-} f(x)\,d\mu_{\partial B_n}(x)
\]
denote  the weighted surface area of the cap. Note that $\prob_f(\partial B_n\cap H^+)=1-s$. 
We use \cite[Equation (71)]{Reitzner2002} to merge \cite[Lemma 6]{Reitzner2002} with Lemma \ref{capradiusest} in the following result. 

\begin{lemma}\label{cap}
Fix $u\in\partial B_n$, $h\in[0,1]$, and let $s:=\int_{\partial B_n\cap H^-} f(x)\,d\mu_{\partial B_n}(x)$ and $r$ denote the weighted surface area and radius, respectively, of the cap $\partial B_n\cap H^-$. Let $\delta>0$ be chosen sufficiently small so that \cite[Lemma 6]{Reitzner2002} holds. Then for any $y\in\partial B_n\cap H^-$, there exists a spherical cap $U_y$ centered at $y$ and an absolute constant $C>0$ such that for all $x\in U_y$,
\begin{align*}
    (1+\delta)^{\frac{3}{n-1}}\bigg[\left(\frac{s}{f(x)\vol_{n-1}(B_{n-1})}\right)^{\frac{1}{n-1}}&-\frac{1}{2(n+1)}\left(\frac{s}{f(x)\vol_{n-1}(B_{n-1})}\right)^{\frac{3}{n-1}}\\
    &\qquad\qquad-C\left(\frac{s}{f(x)\vol_{n-1}(B_{n-1})}\right)^{\frac{5}{n-1}}\bigg]\\
    &\leq r \leq   \\
   (1+\delta)^{\frac{1}{n-1}}\bigg[ \left(\frac{s}{f(x)\vol_{n-1}(B_{n-1})}\right)^{\frac{1}{n-1}}&-\frac{1}{2(n+1)}\left(\frac{s}{f(x)\vol_{n-1}(B_{n-1})}\right)^{\frac{3}{n-1}}\\
   &\qquad\qquad+C\left(\frac{s}{f(x)\vol_{n-1}(B_{n-1})}\right)^{\frac{5}{n-1}}\bigg].
\end{align*}
\end{lemma}

\begin{proof}
Consider the function $g:\R^{n-1}\to\R$ whose graph describes $\partial B_n$ locally at $x(u)$. Then 
\begin{align*}
s&=\int_{\{g(x)\geq h\}}f(x)\sqrt{1+(\nabla g(x))^2}\,dx\\
&=\int_{\partial B_{n-1}}\int_0^{\sqrt{1-h^2}}f(ru)\sqrt{1+(\nabla g(ru))^2}\,r^{n-2}\,dr\,d\mu_{\partial B_n}(u).
\end{align*}
By \cite[Lemma 6]{Reitzner2002}, for all sufficiently small $\delta>0$ and all $x\in U_y$,
{\begin{align}\label{reitzner-est}
    (1+\delta)^{-1}f(x)\mu_{\partial B_n}(\partial B_n\cap H^-) \leq s \leq (1+\delta)^2 f(x)\mu_{\partial B_n}(\partial B_n\cap H^-).
\end{align}
Now we apply Lemma \ref{capradiusest} with $S=\mu_{\partial B_n}(\partial B_n\cap H^-)$} to derive that for all $x\in U_y$ and $k\geq 1$,
\begin{align*}
(1+\delta)^{-\frac{k}{n-1}}\left(\frac{S}{\vol_{n-1}(B_{n-1})}\right)^{\frac{k}{n-1}}
&\leq \left(\frac{s}{f(x)\vol_{n-1}(B_{n-1})}\right)^{\frac{k}{n-1}}\\
&\leq 
(1+\delta)^{\frac{2k}{n-1}}\left(\frac{S}{\vol_{n-1}(B_{n-1})}\right)^{\frac{k}{n-1}}.
\end{align*}
Simple computations finish the proof.
\end{proof}

\noindent The following lemma is also a special case of a result of Reitzner \cite{Reitzner2002} (see also \cite[Lemma 12]{GW2018}).

\begin{lemma}\label{zands}
Let $x(u)$ be the point on $\partial B_n$ with fixed outer unit normal vector $u\in\partial B_n$, and let $z=1-h$ be the distance from $H$ to the supporting hyperplane of $\partial B_n$ at $x(u)$, so that $z\in[0,1]$.  
Then for all sufficiently small $\delta>0$, it holds that 
\begin{equation}\label{Zestimate}
\begin{split}
     &\frac{(1+\delta)^{-\frac{2(n+1)}{n-1}}}{2f(x(u))^{\frac{2}{n-1}}\vol_{n-1}(B_{n-1})^{\frac{2}{n-1}}}\, s^{\frac{2}{n-1}}\leq
    z\leq\frac{ (1+\delta)^{\frac{2n}{n-1}}}{2f(x(u))^{\frac{2}{n-1}}\vol_{n-1}(B_{n-1})^{\frac{2}{n-1}}}\, s^{\frac{2}{n-1}}
\end{split}    
\end{equation}
and
\begin{equation*}
\begin{split}
&(1+\delta)^{-n}2^{\frac{n-3}{2}}f(x(u))\mu_{\partial B_{n-1}}(\partial B_{n-1})z^{\frac{n-3}{2}}\leq\frac{ds}{dz}\leq (1+\delta)^{\frac{n^2+1}{n-1}} 2^{\frac{n-3}{2}}f(x(u))\mu_{\partial B_{n-1}}(\partial B_{n-1})z^{\frac{n-3}{2}}.    
\end{split}    
\end{equation*}
\end{lemma}

\subsection{Random simplices}

The next result is due to Miles \cite[Equation (72)]{Miles1971} for integer moments, which was extended to moments of arbitrary nonnegative real order by   Kabluchko, Temesvari and Th\"ale \cite[Proposition 2.8]{KabluchkoEtAl2019} for general beta (and  beta')  distributions in $\R^n$. The following formulation is the special case $\beta=-1$.

\begin{lemma}\label{Uniform-Moments}
Let $\mathcal{V}_{n,n-1}$ denote the $(n-1)$-dimensional volume of the $(n-1)$-dimensional simplex with vertices $X_1,\ldots,X_{n}$ chosen independently and uniformly from $\partial B_{n-1}$, that is,  $\mathcal{V}_{n,n-1}:=\vol_{n-1}([X_1,\ldots,X_n])$. For all real $m\geq 0$, the $m$th moment of $\mathcal{V}_{n,n-1}$, denoted $\E[\mathcal{V}_{n,n-1}^m]$, is given by
\begin{align*}
    \E[\mathcal{V}_{n,n-1}^m] 
    =((n-1)!)^{-m}\frac{\Gamma\left(\frac{n}{2}(n+m-3)+1\right)}{\Gamma\left(\frac{n}{2}(n-3)+\frac{m(n-1)}{2}+1\right)}\left(\frac{\Gamma\left(\frac{n-1}{2}\right)}{\Gamma\left(\frac{n-1+m}{2}\right)}\right)^n\prod_{i=1}^{n-1}\frac{\Gamma\left(\frac{i+m}{2}\right)}{\Gamma\left(\frac{i}{2}\right)}.
\end{align*}
\end{lemma}

The next result gives estimates for moments of an $(n-1)$-dimensional random simplex with vertices distributed according to the restriction of the density $f$ to a great hypersphere of $\partial B_n$. For the second moment, a more general result for all sufficiently smooth convex bodies in $\R^n$ is due to Grote and Werner \cite[Lemma 3.3]{GW2018}. Our proof is tailored towards the case of the sphere and, as a result, is much simpler.

\begin{lemma}\label{moments}
Fix $h\in[0,1]$. Let $x(u)$ be the point on $\partial B_n$ with fixed outer unit normal vector $u\in\partial B_n$, and let $z=1-h$ be the distance from $H$ to the supporting hyperplane of $\partial B_n$ at $x(u)$. Then for all sufficiently small $\delta>0$ and any $m\geq 0$,
\begin{align*}
&(1+\delta)^{-n}(2z)^{\frac{n^2+n(m-3)-m}{2}}f(x(u))^n\mu_{\partial B_{n-1}}(\partial B_{n-1})^n\E[\mathcal{V}_{n,n-1}^m]\\ &\qquad\leq\int_{\partial B_n\cap H} \cdots \int_{\partial B_n\cap H}\vol_{n-1}([x_1,\ldots,x_n])^m (1-h^2)^{-{n\over 2}}\,d\prob_{f_{\partial B_n\cap H}}(x_1)\cdots d\prob_{f_{\partial B_n\cap H}}(x_n)\\
    &\qquad\qquad\leq (1+\delta)^{n}(2z)^{\frac{n^2+n(m-3)-m}{2}}f(x(u))^n\mu_{\partial B_{n-1}}(\partial B_{n-1})^n\E[\mathcal{V}_{n,n-1}^m].
\end{align*}
\end{lemma}
\begin{proof}
We start by recalling from \cite[Lemma 6]{Reitzner2002} that for all sufficiently small $\delta>0$, there exists some $\lambda>0$ such that 
$$
(1+\delta)^{-1}f(x(u)) \leq f(p) \leq (1+\delta)f(x(u))
$$
for all points $p$ in a spherical cap of radius $\lambda$ around an arbitrary boundary point $x(u)\in\partial B_n$. Thus, 
\begin{align*}
&(1+\delta)^{-n}f(x(u))^n\,\mathcal{I}\\
&\qquad\leq\int_{\partial B_n\cap H} \cdots \int_{\partial B_n\cap H}\vol_{n-1}([x_1,\ldots,x_n])^m (1-h^2)^{-{n\over 2}}\,d\prob_{f_{\partial B_n\cap H}}(x_1)\cdots d\prob_{f_{\partial B_n\cap H}}(x_n)\\
&\qquad\qquad\leq(1+\delta)^n f(x(u))^n\,\mathcal{I}
\end{align*}
with
$$
\mathcal{I} := \int_{\partial B_n\cap H} \cdots \int_{\partial B_n\cap H}\vol_{n-1}([x_1,\ldots,x_n])^m (1-h^2)^{-{n\over 2}}\,d\mu_{\partial B_n\cap H}(x_1)\cdots d\mu_{\partial B_n\cap H}(x_n).
$$
Next, we observe that $\partial B_n\cap H$ is an $(n-2)$-dimensional sphere of radius $\sqrt{1-h^2}$. So, substituting $y_i=x_i/\sqrt{1-h^2}$ for $i=1,\ldots,n$, we obtain
\begin{align*}
    \mathcal{I} &= \int_{\partial B_{n-1}} \cdots \int_{\partial B_{n-1}}\vol_{n-1}([\sqrt{1-h^2}\,y_1,\ldots,\sqrt{1-h^2}\,y_n])^m\\
    &\qquad\qquad\qquad\times(1-h^2)^{-{n\over 2}}\,(1-h^2)^{n(n-2)\over 2}\,d\mu_{\partial B_{n-1}}(y_1)\cdots d\mu_{\partial B_{n-1}}(y_n)\\
    &=(1-h^2)^{{(n-1)m\over 2}-{n\over 2}+{n(n-2)\over 2}}\int_{\partial B_{n-1}} \cdots \int_{\partial B_{n-1}}\vol_{n-1}([y_1,\ldots,y_n])^m\\
    &\hspace{8cm}\times d\mu_{\partial B_{n-1}}(y_1)\cdots d\mu_{\partial B_{n-1}}(y_n)\\
    &=(1-h^2)^{n^2+n(m-3)-m\over 2}\,\mu_{\partial B_{n-1}}(\partial B_{n-1})^n\mathbb{E}[\mathcal{V}_{n,n-1}^m],
\end{align*}
where the equality in the last line follows from the definition of $\mathbb{E}[\mathcal{V}_{n,n-1}^m]$. Since we have assumed $\delta$ to be small, $h$ is close to $1$, which implies that $1-h^2$ is asymptotically equivalent to $2z$. The result thus follows.
\end{proof}

\section{Proof of Theorem \ref{T-Thm}}\label{sec:ProofTfunctional}

\subsection{Step 1: Integral representation}

Choose i.i.d.\ random points $X_1,X_2,\ldots$ on $\partial B_n$ according to the density $f$, and for $N\geq n+1$ define the random polytope $Q_{n,N}^f:=[X_1,\ldots,X_N]$ as the convex hull of $X_1,\ldots,X_N$. Since $Q_{n,N}^f$ is simplicial with probability one, we have
\begin{align*}
\E[T_{a,b}(Q_{n,N}^f)]=\int_{\partial B_n}\cdots\int_{\partial B_n}T_{a,b}([x_1,\ldots,x_N])\mathbbm{1}_{A_{n,N,f}}(x_1,\ldots,x_N)\,d\prob_f(x_1)\cdots d\prob_f(x_N)
\end{align*}
where
\begin{equation*}
    A_{n,N,f} := \{(x_1,\dotsc,x_N) \in (\partial B_n)^N : \,[x_1,\dotsc,x_N] \text{ is simplicial}\}
\end{equation*}
and $(\partial B_n)^m = \prod_{i=1}^m \partial B_n$ for $m\in\mathbb{N}$. Let $\mathcal{E}_{n,N,f}$ denote the event that the origin $o$ lies in the interior of $Q_{n,N}^f$. As in \cite[Lemma 4.3 (ii)]{SW2003} (see also the proof of Corollary 1 in \cite{Affentranger1991}), we find that
\begin{align*}
\prob(\mathcal{E}_{n,N,f}^c) &= \prob(\{o\not\in \interior [X_1,\ldots,X_N]\})\\
&=\prob_f^N(\{(x_1,\ldots,x_N)\in (\partial B_n)^N:\,o\not\in \interior [x_1,\ldots,x_N]\}) \leq e^{-\tilde{c}_0(n,f)N}
\end{align*}
for some positive  constant $\tilde{c}_0(n,f)$. Furthermore, every  simplicial polytope $Q=[x_1,\ldots,x_N]$ satisfies $|\mathcal{F}_{n-1}(Q)|\leq {N\choose n}$ since each facet of $Q$ has $n$ vertices. Since $e^n=\sum_{k=0}^\infty\frac{n^k}{k!}\geq \frac{n^n}{n!}$, we have $n!\geq (n/e)^n$. Using this and the elementary estimate $\frac{N!}{(N-n)!}\leq N^n$, for any simplicial polytope $Q$ we have $|\mathcal{F}_{n-1}(Q)|\leq (eN/n)^n$. 
Hence, by the law of total expectation,
\begin{align}
\E[T_{a,b}(Q_{n,N}^f)]
&=\E[T_{a,b}(Q_{n,N}^f)|\mathcal{E}_{n,N,f}]\prob(\mathcal{E}_{n,N,f})+\E[T_{a,b}(Q_{n,N}^f)|\mathcal{E}_{n,N,f}^c]\prob(\mathcal{E}_{n,N,f}^c)\nonumber\\
&\leq \E[T_{a,b}(Q_{n,N}^f)|\mathcal{E}_{n,N,f}]\prob(\mathcal{E}_{n,N,f}) + \left(\frac{eN}{n}\right)^n\vol_{n-1}(B_{n-1})^b e^{-\tilde{c}_0(n,f)N}.\label{law-total-exp}
    \end{align}
    Inequality \eqref{law-total-exp} follows from the fact that  $Q_{n,N}^f\subset B_n$ is simplicial with probability 1 and since for $a,b\geq 0$,  the inequality
    \begin{align*}
T_{a,b}(Q_{n,N}^f) &= \sum_{F\in\mathcal{F}_{n-1}(Q_{n,N}^f)}\dist(o,\aff(F))^a\vol_{n-1}(F)^b\\
&\leq |\mathcal{F}_{n-1}(Q_{n,N}^f)|\vol_{n-1}(B_{n-1})^b\leq \left(\frac{eN}{n}\right)^n\vol_{n-1}(B_{n-1})^b
    \end{align*}
holds with probability 1. Therefore, the second term in \eqref{law-total-exp} is exponentially decreasing in $N$, while, as we shall see, the first term is essentially of the order $N^{-(b-1)}(c_1(n,b,f)-c_2(n,a,b,f)N^{-\frac{2}{n-1}})$. Thus, the second term is negligible and we will ignore it in the subsequent computations. 

Next, set
 \begin{equation*}
    E_{n,N,f} := \{(x_1,\dotsc,x_N) \in (\partial B_n)^N : o \in \interior [x_1,\dotsc,x_N] \text{ and $[x_1,\dotsc,x_n]$ is simplicial}\}.
\end{equation*}
For points $(x_1,\ldots,x_N)\in E_{n,N,f}$, we can decompose $\R^n$ into the following union of cones with pairwise disjoint interiors:
\[
\R^n = \bigcup_{[x_{j_1},\ldots,x_{j_n}] \in \mathcal{F}_{n-1}(Q_{n,N}^f)} {\rm cone}(x_{j_1},\ldots,x_{j_n}),
\]
where ${\rm cone}(y_1,\ldots,y_m):=\{\sum_{i=1}^m a_i y_i:\, a_i\geq 0, i=1,\ldots,m\}$ denotes the cone spanned by a set of vectors $y_1,\ldots,y_m\in\R^n$. For points $y_1,\ldots,y_n\in\R^n$ whose affine hull is an $(n-1)$-dimensional hyperplane $H(y_1,\ldots,y_n)$, let $H^+(y_1,\ldots,y_n)$ denote the halfspace bounded by $H(y_1,\ldots,y_n)$ which contains the origin. 

For points $x_1,\ldots,x_N\in\R^n$ and a subset of indices $\{j_1,\ldots,j_n\}\subset\{1,\ldots,N\}$, we define a functional $\Phi_{j_1,\ldots,j_n}^{a,b}:(\R^n)^N\to[0,\infty)$ by
\[
\Phi_{j_1,\ldots,j_n}^{a,b}(x_1,\ldots,x_N):=\dist(o,\aff([x_{j_1},\ldots,x_{j_n}]))^a\vol_{n-1}([x_{j_1},\ldots,x_{j_n}])^b
\]
if $o\in\interior[x_1,\ldots,x_N]$ and $\dim([x_{j_1},\ldots,x_{j_n}])=n-1$, and set $\Phi_{j_1,\ldots,j_n}^{a,b}(x_1,\ldots,x_N):=0$ otherwise. Hence, we obtain
\begin{align*}
\E[T_{a,b}(Q_{n,N}^f)&|\mathcal{E}_{n,N,f}]\prob(\mathcal{E}_{n,N,f})\\
&=\int_{\partial B_n}\cdots\int_{\partial B_n}T_{a,b}([x_1,\ldots,x_N])\mathbbm{1}_{E_{n,N,f}}(x_1,\ldots,x_N)\,d\prob_f(x_1)\cdots d\prob_f(x_N)\\
 &=\int_{\partial B_n}\cdots\int_{\partial B_n} \sum_{\{j_1,\ldots,j_n\}\subset\{1,\ldots,N\}}\Phi_{j_1,\ldots,j_n}^{a,b}(x_1,\ldots,x_N)\,d\prob_f(x_1)\cdots d\prob_f(x_N)\\
    &={N\choose n}\int_{\partial B_n}\cdots\int_{\partial B_n} \Phi_{1,\ldots,n}^{a,b}(x_1,\ldots,x_N)\,d\prob_f(x_1)\cdots d\prob_f(x_N).
\end{align*}
Applying  Lemma \ref{BPext}, we get
\begin{equation}\label{estimateexpf}
\begin{split}
    \E[T_{a,b}(Q_{n,N}^f)&|\mathcal{E}_{n,N,f}]\prob(\mathcal{E}_{n,N,f}) = (n-1)!{N\choose n}\int_{\partial B_n}\int_0^1\int_{\partial B_n\cap H}\cdots\int_{\partial B_n\cap H}\\
    &\left[\int_{\partial B_n}\cdots\int_{\partial B_n}\Phi_{1,\ldots,n}^{a,b}(x_1,\ldots,x_N)\,d\prob_f(x_{n+1})\cdots d\prob_f(x_N)\right]\times\\
    &\times\vol_{n-1}([x_1,\ldots,x_n])(1-h^2)^{-{n\over 2}}\, d\mu_{\partial B_n\cap H}(x_1)\cdots d\mu_{\partial B_n\cap H}(x_n)\,dh\,d\mu_{\partial B_n}(u).
\end{split}    
\end{equation}
By definition, $\Phi^{a,b}_{1,\dotsc,n}(x_1,\dotsc,x_N)\neq 0$ if $x_1,\dotsc,x_n$ spans a hyperplane $H$ and $[x_1\dotsc,x_n]\in \mathcal{F}_{n-1}([x_1,\dotsc,x_N])$ and $o\in \interior [x_1,\dotsc,x_N]$. In this case, the value of $\Phi^{a,b}_{1,\dotsc,n}(x_1,\dotsc,x_N)$ only depends on $x_1,\dotsc,x_n$. Since 
\begin{align}\label{main-inclusion}
\big\{(x_{n+1},\ldots,x_N)&\in (\partial B_n)^{N-n}:[x_1,\ldots,x_n]\in\mathcal{F}_{n-1}([x_1,\dotsc,x_N])\text{ and }o\in \interior [x_1,\dotsc,x_N]\big\}\nonumber\\
&\subset \big\{(x_{n+1},\ldots,x_N)\in (\partial B_n)^{N-n} : x_{n+1},\ldots,x_N\in H^+(x_1,\ldots,x_n)\big\} \\
&= (\partial B_n\cap H^+(x_1,\dotsc,x_n))^{N-n}, \nonumber
\end{align}
we obtain
\begin{align}
      \prob_f^{N-n}\big(\{(x_{n+1},\ldots,x_N)\in (\partial B_n)^{N-n}&:[x_1,\ldots,x_n]\in\mathcal{F}_{n-1}([x_1,\dotsc,x_N])\text{ and }o\in \interior [x_1,\dotsc,x_N]\}\big) \nonumber\\
      &\leq \left(\int_{\partial B_n\cap H^+}f(x)\,d\mu_{\partial B_n}(x)\right)^{N-n}=:\prob_f^{N-n}(\partial B_n\cap H^+). \label{prob-correct}
\end{align}
Hence,
\begin{align*}
    &\int_{\partial B_n}\cdots \int_{\partial B_n} \Phi^{a,b}_{1,\ldots,n}(x_1,\ldots,x_N) \, d\prob_f(x_{n+1}) \cdots d\prob_f(x_N)\\
    &\qquad \qquad\qquad\qquad\leq \dist(o,\aff([x_{j_1},\ldots,x_{j_n}]))^a\vol_{n-1}([x_{j_1},\ldots,x_{j_n}])^b\, \prob_f^{N-n}(\partial B_n\cap H^+).
\end{align*}
By Lemma \ref{moments} and the substitution $z=1-h$ where $h=\dist(o,H(x_1,\dotsc,x_n))$, we have  
\begin{align*}
    \E[T_{a,b}(Q_{n,N}^f)&|\mathcal{E}_{n,N,f}]\prob(\mathcal{E}_{n,N,f}) \leq 
    (1+\delta)^{n}2^{\frac{n^2+(n-1)(b-2)-3}{2}}{N\choose n}(n-1)!\mu_{\partial B_{n-1}}(\partial B_{n-1})^n\E[\mathcal{V}_{n,n-1}^{b+1}]\\
    &\qquad\times\int_{\partial B_n}\int_0^1 (1-z)^{a}z^{\frac{n^2+(n-1)(b-2)-3}{2}}
    \prob_f^{N-n}(\partial B_n\cap H^+) f(x(u))^n\,dz\,d\mu_{\partial B_n}(u).
\end{align*}
For $u\in\partial B_n$ and $z\in[0,1]$, set 
\[
\phi_f(u,z):=s(\partial B_n\cap H_z(u)^-)
\]
where $H_z(u):=u^\perp+(1-z)u$. Then by Lemma \ref{zands} we get
\begin{align*}
    \E[T_{a,b}(Q_{n,N}^f)&|\mathcal{E}_{n,N,f}]\prob(\mathcal{E}_{n,N,f})\leq (1+\delta)^{n}2^{\frac{n^2+(n-1)(b-2)-3}{2}}{N\choose n}(n-1)!\mu_{\partial B_{n-1}}(\partial B_{n-1})^n\E[\mathcal{V}_{n,n-1}^{b+1}]\\
    &\qquad\times\int_{\partial B_n}\int_{\phi_f(u,0)}^{\phi_f(u,1)} (1-z(u,s))^{a}z(u,s)^{\frac{n^2+(n-1)(b-2)-3}{2}}(1-s)^{N-n} \\
    &\qquad\times f(x(u))^n\frac{(1+\delta)^n 2^{-\frac{n-3}{2}}}{f(x(u))\mu_{\partial B_{n-1}}(\partial B_{n-1})}z(u,s)^{-\frac{n-3}{2}}\,ds\,d\mu_{\partial B_n}(u)\\
    &=(1+\delta)^{2n}2^{\frac{(n-1)(n+b-2)}{2}}{N\choose n}(n-1)!\mu_{\partial B_{n-1}}(\partial B_{n-1})^{n-1}\E[\mathcal{V}_{n,n-1}^{b+1}]\\
    &\qquad\times\int_{\partial B_n}\int_{0}^{\phi_f(u,1)} (1-z(u,s))^{a}z(u,s)^{\frac{(n-1)(n+b-2)}{2}} (1-s)^{N-n}\\ &\qquad\times f(x(u))^{n-1}\,ds\,d\mu_{\partial B_n}(u),
\end{align*}
where we have used $\phi_f(u,0)=0$ for any $u\in\partial B_n$.

To simplify our estimates, we express the cap height $z$ in terms of its radius $r=r(u,s)$. Then $(1-z)^2=1-r^2$, so $(1-z)^{a}=(1-r^2)^{\frac{a}{2}}$, and applying the inequality $1-\frac{x}{2}-\frac{x^2}{2}\leq \sqrt{1-x}\leq 1-\frac{x}{2}$ for $x\leq 1$ with $x=r^2$ we get
\[
\frac{r^2}{2} \leq z\leq \frac{r^2}{2}+\frac{r^4}{2}.
\]
Thus, with these substitutions we obtain
\begin{align*}
    &\E[T_{a,b}(Q_{n,N}^f)|\mathcal{E}_{n,N,f}]\prob(\mathcal{E}_{n,N,f}) \\
    &\leq (1+\delta)^{2n}2^{\frac{(n-1)(n+b-2)}{2}}{N\choose n}(n-1)!\mu_{\partial B_{n-1}}(\partial B_{n-1})^{n-1}\E[\mathcal{V}_{n,n-1}^{b+1}]\\
    &\times\int_{\partial B_n}\int_{0}^{\phi_f(u,1)}(1-s)^{N-n} 
     (1-r^2)^{\frac{a}{2}}\left(\frac{r^2}{2}+\frac{r^4}{2}\right)^{\frac{(n-1)(n+b-2)}{2}}
      f(x(u))^{n-1}\,ds\,d\mu_{\partial B_n}(u)\\
    &=(1+\delta)^{2n}{N\choose n}(n-1)!\mu_{\partial B_{n-1}}(\partial B_{n-1})^{n-1}\E[\mathcal{V}_{n,n-1}^{b+1}]\\
    &\times \int_{\partial B_n}\int_{0}^{\phi_f(u,1)}(1-s)^{N-n}(1-r^2)^{\frac{a}{2}}r^{(n-1)(n+b-2)}\left(1+\frac{r^2}{2}\right)^{\frac{(n-1)(n+b-2)}{2}}  f(x(u))^{n-1}\,ds\,d\mu_{\partial B_n}(u).
\end{align*}

By the inequalities $1+x \leq \left(1+\frac{x}{k}\right)^k \leq e^x$ for $k\geq 1$ and $|x|\leq k$, and $e^x<1+x+x^2$ for $x<1.79$,  we have  
\begin{align*}
1+\frac{1}{4}(n-1)(n+b-2)r^2&\leq \left(1+\frac{r^2}{2}\right)^{\frac{(n-1)^2}{2}}\\
&\leq e^{\frac{r^2(n-1)^2}{4}}<1+\frac{1}{4}(n-1)(n+b-2)r^2+\frac{1}{16}(n-1)^2(n+b-2)^2r^4.
\end{align*}
Now we split the previous upper bound for the  expectation into two parts as \[
\E[T_{a,b}(Q_{n,N}^f)|\mathcal{E}_{n,N,f}]\prob(\mathcal{E}_{n,N,f})\leq I_1+I_2,
\]
where $I_1$ and $I_2$ are defined by
\begin{align*}
      I_1&:=(1+\delta)^{2n}{N\choose n}(n-1)!\mu_{\partial B_{n-1}}(\partial B_{n-1})^{n-1}\E[\mathcal{V}_{n,n-1}^{b+1}]\int_{\partial B_n}\int_{0}^{\phi_f(u,1)}(1-s)^{N-n} \\
    &\times (1-r^2)^{\frac{a}{2}}r^{(n-1)(n+b-2)}\left(1+\frac{(n-1)(n+b-2)}{4}r^2\right)f(x(u))^{n-1}\,ds\,d\mu_{\partial B_n}(u)
    \intertext{and}
    I_2&:=\frac{1}{16}(1+\delta)^{2n}{N\choose n}(n-1)^2(n+b-2)^2(n-1)!\mu_{\partial B_{n-1}}(\partial B_{n-1})^{n-1}\E[\mathcal{V}_{n,n-1}^{b+1}]\\
    &\times\int_{\partial B_n}\int_{0}^{\phi_f(u,1)}(1-s)^{N-n} (1-r^2)^{\frac{a}{2}}r^{(n-1)(n+b-2)+4}f(x(u))^{n-1} \,ds\,d\mu_{\partial B_n}(u).
\end{align*}
We will see later that $I_2$ is of negligible order, so for now we will focus on $I_1$. 

\subsection{Step 2: Breaking $I_1$ into further terms}

Next, we use Lemma \ref{cap} to estimate the terms in the integrand of $I_1$ which directly involve the radius $r$. As we shall see from the proof that follows, the term involving $r^{(n-1)^2+2}$ is of a smaller order in $N$ than $\NN$,  so we will focus only on the first term involving $r^{(n-1)^2}$. We have
\begin{align*}
I_1&\leq (1+\delta)^{2n}{N\choose n}(n-1)!\mu_{\partial B_{n-1}}(\partial B_{n-1})^{n-1}\E[\mathcal{V}_{n,n-1}^{b+1}]\times \\
    &\times \int_{\partial B_n}\int_{0}^{\phi_f(u,1)} 
    \bigg\{(1+\delta)^{\frac{1}{n-1}}\bigg[ \left(\frac{s}{f(x(u))\vol_{n-1}(B_{n-1})}\right)^{\frac{1}{n-1}}\\
   &-\frac{1}{2(n+1)}\left(\frac{s}{f(x(u))\vol_{n-1}(B_{n-1})}\right)^{\frac{3}{n-1}}
   +C\left(\frac{s}{f(x(u))\vol_{n-1}(B_{n-1})}\right)^{\frac{5}{n-1}}\bigg]\bigg\}^{(n-1)(n+b-2)}\\
    &\times\bigg[1-\bigg\{(1+\delta)^{\frac{3}{n-1}}\bigg( \left(\frac{s}{f(x(u))\vol_{n-1}(B_{n-1})}\right)^{\frac{1}{n-1}}
    -\frac{1}{2(n+1)}\left(\frac{s}{f(x(u))\vol_{n-1}(B_{n-1})}\right)^{\frac{3}{n-1}}\\
   &-C\left(\frac{s}{f(x(u))\vol_{n-1}(B_{n-1})}\right)^{\frac{5}{n-1}}\bigg)\bigg\}^2\bigg]^{\frac{a}{2}}
   (1-s)^{N-n} f(x(u))^{n-1}\,ds\,d\mu_{\partial B_n}(u).
\end{align*}
Writing
\begin{align*}
    {\mu_{\partial B_{n-1}}(\partial B_{n-1})^{n-1}\over\vol_{n-1}(B_{n-1})^{n+b-2}} &=\left({\mu_{\partial B_{n-1}}(\partial B_{n-1})\over \vol_{n-1}(B_{n-1})}\right)^{n-1}{1\over\vol_{n-1}(B_{n-1})^{b-1}}\\
    &= \left({(n-1)\vol_{n-1}(B_{n-1})\over \vol_{n-1}(B_{n-1})}\right)^{n-1}{1\over\vol_{n-1}(B_{n-1})^{b-1}} \\
    &= {(n-1)^{n-1}\over\vol_{n-1}(B_{n-1})^{b-1}},
\end{align*}
we see that the previous  expression is equal to
\begin{align*}
   &(1+\delta)^{3n+b-2}{N\choose n}\frac{(n-1)!(n-1)^{n-1}\E[\mathcal{V}_{n,n-1}^{b+1}]}{\vol_{n-1}(B_{n-1})^{b-1}}\int_{\partial B_n}\int_{0}^{\phi_f(u,1)}s^{n+b-2}(1-s)^{N-n}\\
   &\times\left[1-\frac{1}{2(n+1)}\left(\frac{s}{f(x(u))\vol_{n-1}(B_{n-1})}\right)^{\frac{2}{n-1}}+C\left(\frac{s}{f(x(u))\vol_{n-1}(B_{n-1})}\right)^{\frac{4}{n-1}}\right]^{(n-1)(n+b-2)}\\
   &\times \bigg[1-(1+\delta)^{\frac{6}{n-1}}\bigg\{\left(\frac{s}{f(x(u))\vol_{n-1}(B_{n-1})}\right)^{\frac{1}{n-1}}
    -\frac{1}{2(n+1)}\left(\frac{s}{f(x(u))\vol_{n-1}(B_{n-1})}\right)^{\frac{3}{n-1}}\\
   &-C\left(\frac{s}{f(x(u))\vol_{n-1}(B_{n-1})}\right)^{\frac{5}{n-1}}\bigg\}^2\bigg]^{\frac{a}{2}} f(x(u))^{1-b}\, ds\,d\mu_{\partial B_n}(u).
\end{align*}
This can be estimated from above by 
\begin{align*}
    I_1&\leq (1+\delta)^{3n+b-2}{N\choose n}\frac{(n-1)!(n-1)^{n-1}\E[\mathcal{V}_{n,n-1}^{b+1}]}{\vol_{n-1}(B_{n-1})^{b-1}}\times\\
    &\qquad\times\int_{\partial B_n}\int_{0}^{\phi_f(u,1)}s^{n+b-2}(1-s)^{N-n}\bigg[1-\frac{(n-1)(n+b-2)}{2(n+1)}\left(\frac{s}{f(x(u))\vol_{n-1}(B_{n-1})}\right)^{\frac{2}{n-1}}\\
    &\qquad\qquad\qquad+C_1(n-1)(n+b-2)\left(\frac{s}{f(x(u))\vol_{n-1}(B_{n-1})}\right)^{\frac{4}{n-1}}\bigg]\times\\
    &\qquad\times\left[1-\frac{a}{2}\left\{\left(\frac{s}{f(x(u))\vol_{n-1}(B_{n-1})}\right)^{\frac{2}{n-1}}-\frac{C_2}{n+1}\left(\frac{s}{f(x(u))\vol_{n-1}(B_{n-1})}\right)^{\frac{4}{n-1}}\right\}\right]\times\\
    &\qquad\qquad\qquad\qquad\qquad\qquad\qquad\qquad\qquad\qquad\times f(x(u))^{1-b} ds\,d\mu_{\partial B_n}(u)
\end{align*}
for some new absolute constants $C_1,C_2>0$. Hence, with a new absolute constant $C_3>0$ we obtain 
\begin{align*}
    &(1+\delta)^{-(3n+b-2)}I_1\leq {N\choose n}\frac{(n-1)!(n-1)^{n-1}\E[\mathcal{V}_{n,n-1}^{b+1}]}{\vol_{n-1}(B_{n-1})^{b-1}}\times\\
    &\qquad\times\int_{\partial B_n}\int_{0}^{\phi_f(u,1)}s^{n+b-2}(1-s)^{N-n}\times\\
    &\qquad\times \bigg[1-\frac{(n-1)(n+b-2)+a(n+1)}{2(n+1)}\left(\frac{s}{f(x(u))\vol_{n-1}(B_{n-1})}\right)^{\frac{2}{n-1}}\\
    &\qquad+C_3(n+b-2)\left(a+n\right)\left(\frac{s}{f(x(u))\vol_{n-1}(B_{n-1})}\right)^{\frac{4}{n-1}}\bigg]f(x(u))^{1-b}\,ds\,d\mu_{\partial B_n}(u).
\end{align*}

Set $c_0(n,a,b):=\frac{(n-1)(n+b-2)+a(n+1)}{2(n+1)}$ and $C_0(n,a,b):=C_3(n+b-2)(a+n)$. Expanding the integrals and using the fact that the function $s\mapsto s^{t_1}(1-s)^{t_2}$ with $t_1,t_2>0$ is nonnegative for $s\in[0,1]$, we derive that
\begin{align*}
    &(1+\delta)^{-(3n+b-2)}I_1
    \leq {N\choose n}\frac{(n-1)!(n-1)^{n-1}\E[\mathcal{V}_{n,n-1}^{b+1}]}{\vol_{n-1}(B_{n-1})^{b-1}}\times\\
    &\qquad\times\bigg\{\int_{\partial B_n}\int_0^{\phi_f(u,1)}s^{n+b-2}(1-s)^{N-n}f(x(u))^{1-b}\,ds\,d\mu_{\partial B_n}(u)\\
    &\qquad-\frac{c_0(n,a,b)}{\vol_{n-1}(B_{n-1})^{\frac{2}{n-1}}}\bigg[\int_{\partial B_n}\int_0^1 s^{n+b-2+\frac{2}{n-1}}(1-s)^{N-n}f(x(u))^{1-b-\frac{2}{n-1}}\,ds\,d\mu_{\partial B_n}(u)\\
    &\qquad-\int_{\partial B_n}\int_{\phi_f(u,1)}^1 s^{n+b-2+\frac{2}{n-1}}(1-s)^{N-n}f(x(u))^{1-b-\frac{2}{n-1}}\,ds\,d\mu_{\partial B_n}(u)\bigg]\\
    &\qquad+\frac{C_0(n,a,b)}{\vol_{n-1}(B_{n-1})^{\frac{4}{n-1}}}\int_{\partial B_n}\int_0^{\phi_f(u,1)}s^{n+b-2+\frac{4}{n-1}}(1-s)^{N-n}f(x(u))^{1-b-\frac{4}{n-1}}\,ds\,d\mu_{\partial B_n}(u)\bigg\}.
\end{align*}

This can be estimated from above by\begin{align*}
    &(1+\delta)^{-(3n+b-2)}I_1 \leq {N\choose n}\frac{(n-1)!(n-1)^{n-1}\E[\mathcal{V}_{n,n-1}^{b+1}]}{\vol_{n-1}(B_{n-1})^{b-1}}\times\\
    &\qquad\times\bigg\{ \int_{\partial B_n}f(x)^{1-b}\,d\mu_{\partial B_n}(x)\int_0^1 s^{n+b-2}(1-s)^{N-n}\,ds\,d\mu_{\partial B_n}(u)\\
    &\qquad-\frac{c_0(n,a,b)}{\vol_{n-1}(B_{n-1})^{\frac{2}{n-1}}}\bigg[\int_{\partial B_n}f(x(u))^{1-b-\frac{2}{n-1}}\int_0^1 s^{n+b-2+\frac{2}{n-1}}(1-s)^{N-n}\,ds\,d\mu_{\partial B_n}(u)\\
    &\qquad-\int_{\partial B_n}f(x(u))^{1-b-\frac{2}{n-1}}\int_{\phi_f(u,1)}^1 s^{n+b-2+\frac{2}{n-1}}(1-s)^{N-n}\,ds\,d\mu_{\partial B_n}(u)\bigg]\\
    &\qquad+\frac{C_0(n,a,b)}{\vol_{n-1}(B_{n-1})^{\frac{4}{n-1}}}\int_{\partial B_n}f(x(u))^{1-b-\frac{4}{n-1}}\int_0^1 s^{n+b-2+\frac{4}{n-1}}(1-s)^{N-n}\,ds\,d\mu_{\partial B_n}(u)\bigg\}.
\end{align*}

\subsection{Step 3: Dealing with the third integral}

Since $f$ is continuous and positive on the compact set $\partial B_n$, it attains a positive minimum value $c_{\min}(f):=\min_{u\in\partial B_n}f(x(u))>0$ which may depend on $n$ but not on $N$.  This implies that for any $u\in\partial B_n$,
\begin{align*}
\phi_f(u,1) &= \int_{\partial B_n\cap H(u,1)^-}f(x)\,d\mu_{\partial B_n}(x) \\
&\geq c_{\min}(f)\vol_{n-1}(\partial B_n\cap H(u,1)^-)\\
&=\frac{1}{2}c_{\min}(f)\mu_{\partial B_n}(\partial B_n).
\end{align*}
Hence, the integral in the third summand can be estimated from above by
\begin{align*}
    \int_{\partial B_n}f(x(u))^{1-b-\frac{2}{n-1}} &\int_{\phi_f(u,1)}^1 s^{n+b-2+\frac{2}{n-1}}(1-s)^{N-n}\,ds\,d\mu_{\partial B_n}(u)\\
    &\leq \int_{\partial B_n}f(x(u))^{1-b-\frac{2}{n-1}}\int_{\phi_f(u,1)}^1 (1-s)^{N-n}\,ds\,d\mu_{\partial B_n}(u)\\
    &=\int_{\partial B_n}f(x(u))^{1-b-\frac{2}{n-1}}\cdot\frac{(1-\phi_f(u,1))^{N-n+1}}{N-n+1}\,d\mu_{\partial B_n}(u)\\
    &\leq \frac{\left(1-\frac{1}{2}c_{\min}(f)\mu_{\partial B_n}(\partial B_n)\right)^{N-n+1}}{N-n+1}\int_{\partial B_n}f(x(u))^{1-b-\frac{2}{n-1}}\, d\mu_{\partial B_n}(u).
\end{align*} 
Necessarily, $c_{\min}(f)\in(0,(\mu_{\partial B_n}(\partial B_n))^{-1}]$, for otherwise 
\[
\int_{\partial B_n}f(x)\,d\mu_{\partial B_n}(x)\geq \int_{\partial B_n}c_{\min}(f)\,d\mu_{\partial B_n}(x)  >1,
\]
a contradiction. Therefore,
\[
1-\tfrac{1}{2}c_{\min}(f)\mu_{\partial B_n}(\partial B_n) \in \left[\tfrac{1}{2},1\right).
\]
Since the function $s\mapsto (1-s)^{N-n}$ is decreasing for $s\in[0,1]$, this implies that
\begin{align*}
    \int_{\partial B_n} \int_{\phi_f(u,1)}^1 &s^{n+b-2+\frac{2}{n-1}}(1-s)^{N-n}f(x(u))^{1-b-\frac{2}{n-1}}\,ds\,d\mu_{\partial B_n}(u)\\
    &\leq \frac{2^{-(N-n+1)}}{N-n+1}\int_{\partial B_n}f(x(u))^{1-b-\frac{2}{n-1}}\, d\mu_{\partial B_n}(u).
\end{align*}

\subsection{Step 4: Putting the bounds together}

Setting
\begin{align*}
C_1(n,a,b)&:=\frac{(n-1)!(n-1)^{n-1}\E[\mathcal{V}_{n,n-1}^{b+1}]}{\vol_{n-1}(B_{n-1})^{b-1+\frac{2}{n-1}}}c_0(n,a,b)\\
  C_2(n,a,b,f) &:=  C_1(n,a,b)\int_{\partial B_n}f(x)^{1-b-\frac{2}{n-1}}\,d\mu_{\partial B_n}(x)
\end{align*}
and using the definition $B(v,w)=\int_0^1 t^{v-1}(1-t)^{w-1}\,dt$ of the beta function, we obtain

\begin{align*}
    & (1+\delta)^{-(3n+b-2)}I_1\\
    &\leq
    {N\choose n}(n-1)!(n-1)^{n-1}\cdot\frac{\E[\mathcal{V}_{n,n-1}^{b+1}]\int_{\partial B_n}f(x)^{1-b}\,d\mu_{\partial B_n}(x)}{\vol_{n-1}(B_{n-1})^{b-1}}\cdot B(N-n+1,n+b-1)\\
    &\qquad-{N\choose n}C_2(n,a,b,f)
    B(N-n+1,n+b-1+\tfrac{2}{n-1})\\
    &\qquad+{N\choose n}\frac{2^{-(N-n+1)}}{N-n+1}C_1(n,a,b)\int_{\partial B_n}f(x(u))^{1-b-\frac{2}{n-1}}\, d\mu_{\partial B_n}(u)\\
     &\qquad+{N\choose n}(n-1)!(n-1)^{n-1}C_0(n,a,b)\frac{\E[\mathcal{V}_{n,n-1}^{b+1}]\int_{\partial B_n}f(x)^{1-b-\frac{4}{n-1}}\,d\mu_{\partial B_n}(x)}{\vol_{n-1}(B_{n-1})^{b-1+\frac{4}{n-1}}}\\
     &\qquad\qquad\qquad\times B\left(N-n+1,n+b-1+\tfrac{4}{n-1}\right).
\end{align*}
From the fact that $B(v,w)=\frac{\Gamma(v)\Gamma(w)}{\Gamma(v+w)}$ and by using the asymptotics of the ratio of gamma functions, we find that
$$
\lim_{N\to\infty}{B(N+v,w)\over \Gamma(w)N^{-w}} = 1.
$$
Applying this to $v=-n+1$ and $w=n+b-1$ in the first line of the previous estimate, $w=n+b-1+{2\over n-1}$ in the second line and $w=n+b-1+{4\over n-1}$ in the last line, we derive that the previous estimate is asymptotically equal to 
\begin{align*}
     (1+\delta)^{-(3n+b-2)}I_1
    &\leq
    n^{-1}(n-1)^{n-1}\Gamma(n+b-1)\E[\mathcal{V}_{n,n-1}^{b+1}]\frac{\int_{\partial B_n}f(x)^{1-b}\,d\mu_{\partial B_n}(x)}{\vol_{n-1}(B_{n-1})^{b-1}}N^{-(b-1)}\\
    &\qquad-{\Gamma(n+b-1+{2\over n-1})\over n!}C_2(n,a,b,f)
    N^{-(b-1+{2\over n-1})}\\
    &\qquad+{N^n\over n!}\frac{2^{-(N-n+1)}}{N-n+1}C_1(n,a,b)\int_{\partial B_n}f(x)^{1-b-\frac{2}{n-1}}\,d\mu_{\partial B_n}(x)\\
     &\qquad+C_0(n,a,b)\E[\mathcal{V}_{n,n-1}^{b+1}]n^{-1}(n-1)^{n-1}\Gamma\left(n+b-1+{4\over n-1}\right)\times\\
    &\qquad\qquad\qquad\times\frac{\int_{\partial B_n}f(x)^{1-b-\frac{4}{n-1}}\, d\mu_{\partial B_n}(x)}{\vol_{n-1}(B_{n-1})^{n+\frac{4}{n-1}}}N^{-(b-1+{4\over n-1})},
\end{align*}
where we also used that ${N\choose n}$ is asymptotically equivalent to $N^n/n!$. Now, observe that the last two summands are of negligible order compared to $\NN$. Finally, we send $\delta$ to 0,  which establishes the upper bound for $I_1$. The upper bound for $I_2$ is handled in the very same way, and it turns out that $I_2$ is of the order $N^{-(b-1+\frac{4}{n-1})}$. This proves one direction of Theorem \ref{T-Thm}.

The lower bound for $\E [T_{a,b}(Q_{n,N}^f)]$ is provided by almost exactly the same method. The main difference is that we must obtain an inclusion in the opposite direction of that in  \eqref{main-inclusion}. Using the same notation as there, we note that
\begin{align*}
    &\phantom{=}\{(x_{n+1},\ldots,x_N)\in(\partial B_n)^{N-n}:\,[x_1,\ldots,x_n]\in\mathcal{F}_{n-1}([x_1,\ldots,x_N])\text{ and }o\in\interior[x_1,\ldots,x_N]\}\\
    &=\{(x_{n+1},\ldots,x_N)\in(\partial B_n)^{N-n}:\,[x_1,\ldots,x_n]\in\mathcal{F}_{n-1}([x_1,\ldots,x_N])\}\\
    &\setminus \{(x_{n+1},\ldots,x_N)\in(\partial B_n)^{N-n}:\,[x_1,\ldots,x_n]\in\mathcal{F}_{n-1}([x_1,\ldots,x_N])\text{ and }o\not\in\interior[x_1,\ldots,x_N]\}\\
     &\supset\{(x_{n+1},\ldots,x_N)\in(\partial B_n)^{N-n}:\,[x_1,\ldots,x_n]\in\mathcal{F}_{n-1}([x_1,\ldots,x_N])\}\\
    &\setminus \{(x_{n+1},\ldots,x_N)\in(\partial B_n)^{N-n}:\,o\not\in\interior[x_1,\ldots,x_N]\}\\
     &\supset(\partial B_n\cap H^+)^{N-n}
    \setminus \{(x_{n+1},\ldots,x_N)\in(\partial B_n)^{N-n}:\,o\not\in\interior[x_1,\ldots,x_N]\}.
\end{align*}
Therefore, using the general inequality $\prob(A\setminus B)=\prob(A)-\prob(A\cap B)\geq \prob(A)-\prob(B)$, which holds for any probability measure $\prob$ and any events $A$ and $B$, we obtain
\begin{align*}
    \prob_f^{N-n}\big(\{(x_{n+1},\ldots,x_N)&\in(\partial B_n)^{N-n}:\,[x_1,\ldots,x_n]\in\mathcal{F}_{n-1}([x_1,\ldots,x_N])\text{ and }o\in\interior[x_1,\ldots,x_N]\}\big)\\
    &\geq \prob_f^{N-n}(\partial B_n\cap H^+)-e^{-\hat{c}_0(n,f)(N-n)}
\end{align*}
for some positive constant $\hat{c}_0(n,f)$. The rest of the proof of the lower bound now proceeds in the same fashion as before, but this time using the opposite inequalities provided in the preliminary geometric lemmas. \qed

\section*{Acknowledgments}

\noindent BL was supported in part by National Natural Science Foundation of China (51202480).
MR was supported in part by the Zuckerman STEM Leadership Program. Part of this work was completed while third named author was in residence at the Institute for Computational and Experimental Research in Mathematics in Providence, RI, during the Harmonic Analysis and Convexity program; this residency was supported by the National Science Foundation under Grant DMS-1929284. CT was supported by the DFG priority program SPP 2265 \textit{Random Geometric Systems}.

We would like to thank the referee for valuable comments that helped to improve the quality of the article. 

\bibliographystyle{plain}
\bibliography{main}

\begin{thebibliography}{10}

\bibitem{Affentranger1991}
F.~Affentranger.
\newblock The convex hull of random points with spherically symmetric
  distributions.
\newblock {\em Rendiconti del Seminario Matematico - Politecnico di Torino},
  49:359--383, 1991.

\bibitem{BaranySurvey}
I.~B\'{a}r\'{a}ny.
\newblock Random polytopes, convex bodies, and approximation.
\newblock In {\em Stochastic geometry}, volume 1892 of {\em Lecture Notes in
  Mathematics}, pages 77--118. Springer, Berlin, 2007.

\bibitem{BH-2022}
F.~Besau and S.~Hoehner.
\newblock An intrinsic volume metric for the class of convex bodies in
  $\mathbb{R}^n$.
\newblock {\em Communications in Contemporary Mathematics (to appear)}, 2023.

\bibitem{BHK}
F.~Besau, S.~Hoehner, and G.~Kur.
\newblock Intrinsic and {D}ual {V}olume {D}eviations of {C}onvex {B}odies and
  {P}olytopes.
\newblock {\em International Mathematics Research Notices},
  2021(22):17456--17513, 2021.

\bibitem{BoroczkyCsikos}
K.~J. B\"or\"oczky and B.~Csik\'os.
\newblock Approximation of smooth convex bodies by circumscribed polytopes with
  respect to the surface area.
\newblock {\em Abhandlungen aus dem Mathematischen Seminar der Universit\"at
  Hamburg}, 79:229--264, 2009.

\bibitem{BoroczkyReitzner04}
K.~J. B\"or\"oczky and M.~Reitzner.
\newblock Approximation of smooth convex bodies by random polytopes.
\newblock {\em The Annals of Applied Probability}, 14:239--273, 2004.

\bibitem{GKT-2021}
T.~Godland, Z.~Kabluchko, and C.~Th\"{a}le.
\newblock Beta-star polytopes and hyperbolic stochastic geometry.
\newblock {\em Advances in Mathematics}, 404(part A):Paper No. 108382, 69,
  2022.

\bibitem{GW2018}
J.~Grote and E.~Werner.
\newblock Approximation of smooth convex bodies by random polytopes.
\newblock {\em Electronic Journal of Probability}, 23:1--21, 2018.

\bibitem{HK-DCG}
S.~Hoehner and G.~Kur.
\newblock A {C}oncentration {I}nequality for {R}andom {P}olytopes,
  {D}irichlet-{V}oronoi {T}iling {N}umbers and the {G}eometric {B}alls and
  {B}ins {P}roblem.
\newblock {\em Discrete \& Computational Geometry}, 65(3):730--763, 2021.

\bibitem{HSW}
S.~Hoehner, C.~Sch\"utt, and E.~Werner.
\newblock The {S}urface {A}rea {D}eviation of the {E}uclidean {B}all and a
  {P}olytope.
\newblock {\em Journal of Theoretical Probability}, 31:244--267, 2018.

\bibitem{HugSurvey}
D.~Hug.
\newblock Random polytopes.
\newblock In {\em Stochastic geometry, spatial statistics and random fields},
  volume 2068 of {\em Lecture Notes in Mathematics}, pages 205--238. Springer,
  Heidelberg, 2013.

\bibitem{HLYZ2005}
D.~Hug, E.~Lutwak, D.~Yang, and G.~Zhang.
\newblock On the ${L}_p$ {M}inkowski {P}roblem for {P}olytopes.
\newblock {\em Discrete \& Computational Geometry}, 33(4):699--715, 2005.

\bibitem{HugMunsoniusReitzner}
D.~Hug, G.~O. Munsonius, and M.~Reitzner.
\newblock Asymptotic mean values of {G}aussian polytopes.
\newblock {\em Beitr\"{a}ge zur Algebra und Geometrie}, 45(2):531--548, 2004.

\bibitem{KMTT-2019}
Z.~Kabluchko, A.~Marynych, D.~Temesvari, and C.~Th\"ale.
\newblock Cones generated by random points on half-spheres and convex hulls of
  poisson point processes.
\newblock {\em Probability Theory and Related Fields}, 175:1021--1061, 2019.

\bibitem{KabluchkoEtAl2019}
Z.~Kabluchko, D.~Temesvari, and C.~Th\"ale.
\newblock Expected intrinsic volumes and facet numbers of random
  beta-polytopes.
\newblock {\em Mathematische Nachrichten}, 292(1):79--105, 2019.

\bibitem{Kur2017}
G.~Kur.
\newblock Approximation of the {E}uclidean ball by polytopes with a restricted
  number of facets.
\newblock {\em Studia Mathematica}, 251(2):111--133, 2020.

\bibitem{LSW}
M.~Ludwig, C.~Sch\"utt, and E.~M. Werner.
\newblock Approximation of the {E}uclidean ball by polytopes.
\newblock {\em Studia Mathematica}, 173:1--18, 2006.

\bibitem{Lutwak93}
E.~Lutwak.
\newblock The {B}runn-{M}inkowski-{F}irey theory. {I}. {M}ixed volumes and the
  {M}inkowski problem.
\newblock {\em Journal of Differential Geometry}, 38(1):131 -- 150, 1993.

\bibitem{Miles1971}
R.~E. Miles.
\newblock Isotropic random simplices.
\newblock {\em Advances in Applied Probability}, 3:353--382, 1971.

\bibitem{Muller1990}
J.~S. M\"uller.
\newblock Approximation of a {B}all by {R}andom {P}olytopes.
\newblock {\em Journal of Approximation Theory}, 63:198--209, 1990.

\bibitem{Reitzner2002}
M.~Reitzner.
\newblock Random points on the boundary of smooth convex bodies.
\newblock {\em Transactions of the American Mathematical Society},
  354:2243--2278, 2002.

\bibitem{ReitznerSurvey}
M.~Reitzner.
\newblock Random polytopes.
\newblock In {\em New perspectives in stochastic geometry}, pages 45--76.
  Oxford University Press, Oxford, 2010.

\bibitem{RenyiSulanke}
A.~R\'{e}nyi and R.~Sulanke.
\newblock \"{U}ber die konvexe {H}\"{u}lle von {$n$} zuf\"{a}llig gew\"{a}hlten
  {P}unkten.
\newblock {\em Zeitschrift f\"{u}r Wahrscheinlichkeitstheorie und Verwandte
  Gebiete}, 2:75--84 (1963), 1963.

\bibitem{SW2003}
C.~Sch{\"u}tt and E.~Werner.
\newblock {\em Polytopes with {V}ertices {C}hosen {R}andomly from the
  {B}oundary of a {C}onvex {B}ody}, volume 1807 of {\em Lecture Notes in
  Mathematics}, pages 241--422.
\newblock Springer, Berlin, Heidelberg, 2003.

\bibitem{Stancu2002}
A.~Stancu.
\newblock On the discrete planar ${L}_0$-{M}inkowski problem.
\newblock {\em Advances in Mathematics}, 167:160--174, 2002.

\bibitem{Stancu2003}
A.~Stancu.
\newblock On the number of solutions to the discrete two-dimensional
  ${L}_0$-{M}inkowski problem.
\newblock {\em Advances in Mathematics}, 180(1):290--323, 2003.

\bibitem{WieackerThesis}
J.~A. Wieacker.
\newblock {\em Einige {P}robleme der polyedrischen {A}pproximation}.
\newblock PhD thesis, Albert-Ludwigs-Universit\"at, Freiburg im Breisgau, 1978.

\bibitem{Zahle1990}
M.~Z\"ahle.
\newblock A kinematic formula and moment measures of random sets.
\newblock {\em Mathematische Nachrichten}, 149:325--340, 1990.

\bibitem{ZhuIUMJ}
G.~Zhu.
\newblock The ${L}_p$ {M}inkowski {P}roblem for {P}olytopes for $p<0$.
\newblock {\em Indiana University Mathematics Journal}, 66(4):1333--1350, 2017.

\end{thebibliography}

\end{document}